\documentclass[a4paper,11pt]{article}
\usepackage[T1]{fontenc}
\usepackage{lmodern}
\usepackage{amsmath,amsthm,amssymb,amsfonts}
\usepackage[dvipsnames,svgnames,table]{xcolor}
\usepackage{pgf,tikz}
\usetikzlibrary{calc}
\usetikzlibrary{arrows.meta}
\usepackage{enumitem}
\usepackage{fullpage}
\usepackage{mathtools}
\usepackage{nicematrix}
\usepackage{booktabs}
\usepackage[title]{appendix}
\usepackage{thmtools, thm-restate}
\usepackage[colorlinks=true,linkcolor=RoyalBlue,urlcolor=RoyalBlue,citecolor=PineGreen]{hyperref}
\usepackage[nameinlink]{cleveref}

\newtheorem{theorem}{Theorem}[section]

\newtheorem{lemma}[theorem]{Lemma}

\newtheorem{conjecture}[theorem]{Conjecture}

\newtheorem{claim}{Claim}

\theoremstyle{definition}

\theoremstyle{remark}

\newcommand{\suchthat}{\;\ifnum\currentgrouptype=16 \middle\fi|\;}

\DeclareMathOperator{\rank}{rank}

\colorlet{colbg}{white}
\colorlet{colfg}{black}
\colorlet{colgraphv}{colfg!75!white}
\colorlet{colgraphe}{colfg!55!white}
\colorlet{colG}{DarkSeaGreen}
\definecolor{colR}{HTML}{CC6677}
\definecolor{colO}{HTML}{DDCC77}
\definecolor{colB}{HTML}{6699CC}
\colorlet{colY}{Gold!90!black}

\tikzstyle{vertex}=[fill=colgraphv,circle,inner sep=0pt, minimum size=4pt]
\tikzstyle{rvertex}=[fill=red,circle,inner sep=0pt, minimum size=4pt]
\tikzstyle{edge}=[line width=1.5pt,colgraphe]
\tikzstyle{redge}=[line width=1.5pt,red]
\tikzstyle{dedge}=[edge, -{LaTeX[round,length=8pt]}]

\definecolor{myorange}{RGB}{239,128,108}
\definecolor{myblue}{RGB}{154,210,201}
\definecolor{mypurple}{RGB}{117,0,235}
\tikzstyle{oedge}=[line width=1.5pt,myorange]
\tikzstyle{bedge}=[line width=1.5pt,myblue]
\tikzstyle{pedge}=[line width=1.5pt,mypurple]

\tikzstyle{labelsty}=[font=\scriptsize]

\title{A counter-example to Baranyai's combinatorial characterisation for 3-rigidity}
\author{Sean Dewar\thanks{Department of Computer Science, KU Leuven. E-mail: \texttt{sean.dewar@kuleuven.be}}
}

\begin{document}
\date{}
\maketitle

\begin{abstract}
    In \cite{baranyai2026genericrigiditygraphs},
    Baranyai described a necessary combinatorial characterisation of graph rigidity for dimension 3.
    In this short note we provide a counter-example to the converse of the condition.
    Additionally, we provide an alternative proof to the Baranyai's necessary condition.    
\end{abstract}

\section{Introduction}

A long standing problem in the field of rigidity theory is to find a combinatorial characterisation of 3-dimensional graph rigidity.
In \cite{baranyai2026genericrigiditygraphs}, Baranyai provided a new necessary condition for rigidity:

\begin{theorem}[Baranyai \cite{baranyai2026genericrigiditygraphs}]\label{mainthm}
    Let $G=(V,E)$ be a minimally 3-rigid graph with at least 4 vertices. Then for any edge $e  \in E$, there exists a partition $(S_1,S_2,S_3)$ of $E$ such that the following holds:
    \begin{enumerate}
        \item $|S_i|=|V|-i$ for each $i \in \{1,2,3\}$.
        \item $e \in S_1$.
        \item Each of the graphs
        \begin{equation*}
            (V, S_1 \cup S_2) , \qquad (V, S_1 \cup S_3) /e, \qquad (V, S_2 \cup S_3 \cup e)/e
        \end{equation*}
        is minimally 2-rigid.
    \end{enumerate}
\end{theorem}

Baranyai additionally claimed a proof for the converse of \Cref{mainthm} (see \cite[Thm 1]{baranyai2026genericrigiditygraphs}).
In this short note we provide a counter-example to this claim.
We additionally provide an alternative proof of \Cref{mainthm} that solely uses arguments and techniques from linear algebra.

\section{Background on rigidity}

We first begin with a brief covering of the core concepts of rigidity theory.
We refer an interested reader to \cite{GraverServatius} for more details.

A \emph{($d$-dimensional) framework} is a pair $(G,p)$ where $G=(V,E)$ is a (simple finite) graph and $p: V \rightarrow \mathbb{R}^d$ is a map with $p(v) = (p_i(v))_{i \in [d]}$ for each $v \in V$.
The map $p$ is also said to be a \emph{($d$-dimensional) realisation} of $G$.
We say that a framework $(G,p)$ is \emph{rigid} if there exists $\varepsilon >0$ such that any other $d$-dimensional framework $(G,q)$ satisfying the equations
\begin{align*}
    \|q(v)-q(w)\|&=\|p(v)-p(w)\| &\text{ for all } vw \in E, \\
    \|p(v)-q(v)\|&< \varepsilon &\text{ for all } v \in V
\end{align*}
also satisfies $\|q(v)-q(w)\|=\|p(v)-p(w)\|$ for each pair of vertices $v,w \in V$.

As rigidity is difficult to check \cite{Abbot},
we opt to linearise the problem as follows.
The \emph{rigidity matrix} $R(G,p)$ of a $d$-dimensional framework $(G,p)$ is now the $|E| \times d|V|$ matrix with rows indexed by edges $e \in E$ and columns indexed by pairs $(v,i) \in V \times [d]$ such that 
\begin{equation*}
    R(G,p)_{e,(v,i)}
    := 
    \begin{cases}
        p_i(v) - p_i(w) &\text{if } e = vw,\\
        0 &\text{otherwise}.
    \end{cases}
\end{equation*}
So long as the graph $G$ has at least $d$ vertices,
we say that $(G,p)$ is \emph{infinitesimally rigid} if $\rank R(G,p) = d|V|-\binom{d+1}{2}$;
if a graph has less than $d$ vertices then we say that it is infinitesimally rigid if $G$ is a complete graph and the points $p(v)$, $v \in V$, are affinely independent.
With this, any infinitesimally rigid framework is rigid \cite{AsimowRothI,AsimowRothII},
although the converse is not true.

Infinitesimal rigidity plays especially well with \emph{generic} realisations: any realisation $p$ where the coordinate multiset $\{p_i(v) :v \in V, \ i \in [d]\}$ forms an algebraically independent set over the rational numbers $\mathbb{Q}$.
It was shown by Asimow and Roth \cite{AsimowRothI,AsimowRothII} that either every generic $d$-dimensional realisation of a graph is infinitesimally rigid, or every $d$-dimensional realisation (including non-generic ones) of $G$ is not infinitesimally rigid.
This inspires the following definition:
and we say that the graph $G$ is \emph{$d$-rigid} if there exists an infinitesimally rigid $d$-dimensional framework $(G,p)$.
We additionally say that $G$ is \emph{minimally $d$-rigid} if it is $d$-rigid but fails to be $d$-rigid if we delete any edge.

Combinatorial characterisations for minimal $d$-rigidity exist for small values of $d$.
The case of $d=1$ is simple: a graph is minimally 1-rigid if and only if it is a tree.
For the case of $d=2$, we introduce the following terminology.
Given positive integers $k,\ell$, we say a graph $G=(V,E)$ is \emph{$(k,\ell)$-sparse} if every subgraph $G'=(V',E')$ with $|V'| \geq k$ satisfies the inequality $|E'|\leq k|V'|-\ell$,
and \emph{$(k,\ell)$-tight} if $G$ is $(k,\ell)$-sparse and satisfies the equality $|E|=k|V|-\ell$.
With relatively little effort, one can use the rigidity matrix to show that every minimally $d$-rigid graph is $(d,\binom{d+1}{2})$-tight.
The converse of this is also true for $d=2$ case.

\begin{theorem}[Geiringer-Laman theorem \cite{pollaczekgeiringer1927,Laman1970}]
    A graph is minimally 2-rigid if and only if it is $(2,3)$-tight, or it is a single vertex.
\end{theorem}

It is tempting to believe that every $(d,\binom{d+1}{2})$-tight must by minimally $d$-rigid, but this is unfortunately false (see the infamous `double banana' in the next section for a classic counter-example for $d=3$).
To this day, no known combinatorial characterisation for minimal $d$-rigidity exists for any $d \geq 3$.

\section{Counter-example to the converse of \texorpdfstring{\Cref{mainthm}}{Theorem 1.1}}\label{sec:counter}

Let $G=(V,E)$ be the double banana graph pictured below:
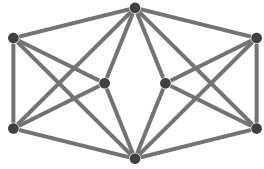
\begin{center}
\begin{tikzpicture}[scale=2]          
	\node[vertex] (r1) at (0,0.5) {};
	\node[vertex] (r2) at (0,-0.5) {};
	\node[vertex] (a1) at (-0.8,0.3) {};
	\node[vertex] (a2) at (0.8,0.3) {};
	\node[vertex] (b1) at (-0.2,0) {};
	\node[vertex] (b2) at (0.2,0) {};
	\node[vertex] (c1) at (-0.8,-0.3) {};
	\node[vertex] (c2) at (0.8,-0.3) {};
	\draw[edge] (r1)edge(a1) (r1)edge(b1) (r1)edge(c1) (r1)edge(a2) (r1)edge(b2) (r1)edge(c2);
	\draw[edge] (r2)edge(a1) (r2)edge(b1) (r2)edge(c1) (r2)edge(a2) (r2)edge(b2) (r2)edge(c2);
	\draw[edge] (a1)edge(b1) (a1)edge(c1) (b1)edge(c1);
	\draw[edge] (a2)edge(b2) (a2)edge(c2) (b2)edge(c2);
\end{tikzpicture}
\end{center}
We are now going to prove that $G$ satisfies the necessary conditions for minimal 3-rigidity outlined in \Cref{mainthm} by showing that each edge allows for a suitable choice of $S_1,S_2,S_3$.

We first choose $e$ to be the edge highlighted in red:
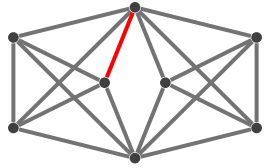
\begin{center}
\begin{tikzpicture}[scale=2]          
	\node[vertex] (r1) at (0,0.5) {};
	\node[vertex] (r2) at (0,-0.5) {};
	\node[vertex] (a1) at (-0.8,0.3) {};
	\node[vertex] (a2) at (0.8,0.3) {};
	\node[vertex] (b1) at (-0.2,0) {};
	\node[vertex] (b2) at (0.2,0) {};
	\node[vertex] (c1) at (-0.8,-0.3) {};
	\node[vertex] (c2) at (0.8,-0.3) {};
	\draw[edge] (r1)edge(a1) (r1)edge(c1) (r1)edge(a2) (r1)edge(b2) (r1)edge(c2);
	\draw[edge] (r2)edge(a1) (r2)edge(b1) (r2)edge(c1) (r2)edge(a2) (r2)edge(b2) (r2)edge(c2);
	\draw[edge] (a1)edge(b1) (a1)edge(c1) (b1)edge(c1);
	\draw[edge] (a2)edge(b2) (a2)edge(c2) (b2)edge(c2);
    \draw[redge] (r1)edge(b1);
\end{tikzpicture}
\end{center}
Here we choose the following decomposition into $S_1$ (orange), $S_2$ (blue) and $S_3$ (purple):
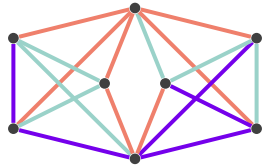
\begin{center}
\begin{tikzpicture}[scale=2]          
	\node[vertex] (r1) at (0,0.5) {};
	\node[vertex] (r2) at (0,-0.5) {};
	\node[vertex] (a1) at (-0.8,0.3) {};
	\node[vertex] (a2) at (0.8,0.3) {};
	\node[vertex] (b1) at (-0.2,0) {};
	\node[vertex] (b2) at (0.2,0) {};
	\node[vertex] (c1) at (-0.8,-0.3) {};
	\node[vertex] (c2) at (0.8,-0.3) {};
	\draw[oedge] (r1)edge(a1);
    \draw[oedge] (r1)edge(b1);
    \draw[oedge] (r1)edge(c1);
    \draw[oedge] (r1)edge(a2);
    \draw[bedge] (r1)edge(b2);
    \draw[oedge] (r1)edge(c2);
	\draw[bedge] (r2)edge(a1);
    \draw[oedge] (r2)edge(b1);
    \draw[pedge] (r2)edge(c1);
    \draw[pedge] (r2)edge(a2);
    \draw[oedge] (r2)edge(b2);
    \draw[pedge] (r2)edge(c2);
	\draw[bedge] (a1)edge(b1);
    \draw[pedge] (a1)edge(c1);
    \draw[bedge] (b1)edge(c1);
	\draw[bedge] (a2)edge(b2);
    \draw[bedge] (a2)edge(c2);
    \draw[pedge] (b2)edge(c2);
\end{tikzpicture}
\end{center}
This gives, from left to right, the graphs $(V, S_1 \cup S_2) , (V, S_1 \cup S_3) /e, (V, S_2 \cup S_3 \cup e)/e$ as follows, with the vertex formed by contracting $e$ highlighted in red:
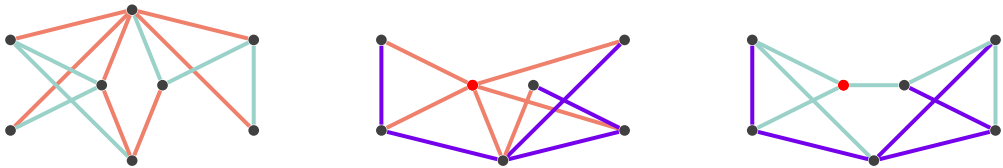
\begin{center}
\begin{tikzpicture}[scale=2]          
	\node[vertex] (r1) at (0,0.5) {};
	\node[vertex] (r2) at (0,-0.5) {};
	\node[vertex] (a1) at (-0.8,0.3) {};
	\node[vertex] (a2) at (0.8,0.3) {};
	\node[vertex] (b1) at (-0.2,0) {};
	\node[vertex] (b2) at (0.2,0) {};
	\node[vertex] (c1) at (-0.8,-0.3) {};
	\node[vertex] (c2) at (0.8,-0.3) {};
	\draw[oedge] (r1)edge(a1);
    \draw[oedge] (r1)edge(b1);
    \draw[oedge] (r1)edge(c1);
    \draw[oedge] (r1)edge(a2);
    \draw[bedge] (r1)edge(b2);
    \draw[oedge] (r1)edge(c2);
	\draw[bedge] (r2)edge(a1);
    \draw[oedge] (r2)edge(b1);
    \draw[oedge] (r2)edge(b2);
	\draw[bedge] (a1)edge(b1);
    \draw[bedge] (b1)edge(c1);
	\draw[bedge] (a2)edge(b2);
    \draw[bedge] (a2)edge(c2);
\end{tikzpicture}\qquad\qquad
\begin{tikzpicture}[scale=2]          
	\node[rvertex] (r1) at (-0.2,0) {};
	\node[vertex] (r2) at (0,-0.5) {};
	\node[vertex] (a1) at (-0.8,0.3) {};
	\node[vertex] (a2) at (0.8,0.3) {};
	\node[rvertex] (b1) at (-0.2,0) {};
	\node[vertex] (b2) at (0.2,0) {};
	\node[vertex] (c1) at (-0.8,-0.3) {};
	\node[vertex] (c2) at (0.8,-0.3) {};
	\draw[oedge] (r1)edge(a1);
    \draw[oedge] (r1)edge(c1);
    \draw[oedge] (r1)edge(a2);
    \draw[oedge] (r1)edge(c2);
    \draw[oedge] (r2)edge(b1);
    \draw[pedge] (r2)edge(c1);
    \draw[pedge] (r2)edge(a2);
    \draw[oedge] (r2)edge(b2);
    \draw[pedge] (r2)edge(c2);
    \draw[pedge] (a1)edge(c1);
    \draw[pedge] (b2)edge(c2);
\end{tikzpicture}\qquad\qquad
\begin{tikzpicture}[scale=2]          
	\node[rvertex] (r1) at (-0.2,0) {};
	\node[vertex] (r2) at (0,-0.5) {};
	\node[vertex] (a1) at (-0.8,0.3) {};
	\node[vertex] (a2) at (0.8,0.3) {};
	\node[rvertex] (b1) at (-0.2,0) {};
	\node[vertex] (b2) at (0.2,0) {};
	\node[vertex] (c1) at (-0.8,-0.3) {};
	\node[vertex] (c2) at (0.8,-0.3) {};
    \draw[bedge] (r1)edge(b2);
	\draw[bedge] (r2)edge(a1);
    \draw[pedge] (r2)edge(c1);
    \draw[pedge] (r2)edge(a2);
    \draw[pedge] (r2)edge(c2);
	\draw[bedge] (a1)edge(b1);
    \draw[pedge] (a1)edge(c1);
    \draw[bedge] (b1)edge(c1);
	\draw[bedge] (a2)edge(b2);
    \draw[bedge] (a2)edge(c2);
    \draw[pedge] (b2)edge(c2);
\end{tikzpicture}
\end{center}
We note here that each of these graphs are $(2,3)$-tight, and hence minimally 2-rigid.

Now choose $e$ to be the edge highlighted in red:
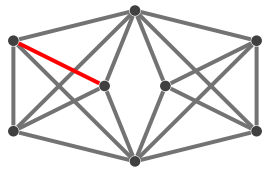
\begin{center}
\begin{tikzpicture}[scale=2]          
	\node[vertex] (r1) at (0,0.5) {};
	\node[vertex] (r2) at (0,-0.5) {};
	\node[vertex] (a1) at (-0.8,0.3) {};
	\node[vertex] (a2) at (0.8,0.3) {};
	\node[vertex] (b1) at (-0.2,0) {};
	\node[vertex] (b2) at (0.2,0) {};
	\node[vertex] (c1) at (-0.8,-0.3) {};
	\node[vertex] (c2) at (0.8,-0.3) {};
	\draw[edge] (r1)edge(a1);
    \draw[edge] (r1)edge(b1);
    \draw[edge] (r1)edge(c1);
    \draw[edge] (r1)edge(a2);
    \draw[edge] (r1)edge(b2);
    \draw[edge] (r1)edge(c2);
	\draw[edge] (r2)edge(a1);
    \draw[edge] (r2)edge(b1);
    \draw[edge] (r2)edge(c1);
    \draw[edge] (r2)edge(a2);
    \draw[edge] (r2)edge(b2);
    \draw[edge] (r2)edge(c2);
	\draw[redge] (a1)edge(b1);
    \draw[edge] (a1)edge(c1);
    \draw[edge] (b1)edge(c1);
	\draw[edge] (a2)edge(b2);
    \draw[edge] (a2)edge(c2);
    \draw[edge] (b2)edge(c2);
\end{tikzpicture}
\end{center}
Here we choose the following decomposition into $S_1$ (orange), $S_2$ (blue) and $S_3$ (purple):
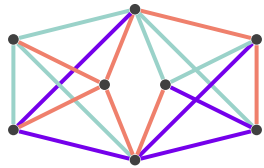
\begin{center}
\begin{tikzpicture}[scale=2]          
	\node[vertex] (r1) at (0,0.5) {};
	\node[vertex] (r2) at (0,-0.5) {};
	\node[vertex] (a1) at (-0.8,0.3) {};
	\node[vertex] (a2) at (0.8,0.3) {};
	\node[vertex] (b1) at (-0.2,0) {};
	\node[vertex] (b2) at (0.2,0) {};
	\node[vertex] (c1) at (-0.8,-0.3) {};
	\node[vertex] (c2) at (0.8,-0.3) {};
	\draw[bedge] (r1)edge(a1);
    \draw[oedge] (r1)edge(b1);
    \draw[pedge] (r1)edge(c1);
    \draw[oedge] (r1)edge(a2);
    \draw[bedge] (r1)edge(b2);
    \draw[bedge] (r1)edge(c2);
	\draw[bedge] (r2)edge(a1);
    \draw[oedge] (r2)edge(b1);
    \draw[pedge] (r2)edge(c1);
    \draw[pedge] (r2)edge(a2);
    \draw[oedge] (r2)edge(b2);
    \draw[pedge] (r2)edge(c2);
	\draw[oedge] (a1)edge(b1);
    \draw[bedge] (a1)edge(c1);
    \draw[oedge] (b1)edge(c1);
	\draw[bedge] (a2)edge(b2);
    \draw[oedge] (a2)edge(c2);
    \draw[pedge] (b2)edge(c2);
\end{tikzpicture}
\end{center}
This gives, from left to right, the graphs $(V, S_1 \cup S_2) , (V, S_1 \cup S_3) /e, (V, S_2 \cup S_3 \cup e)/e$ as follows, with the vertex formed by contracting $e$ highlighted in red:
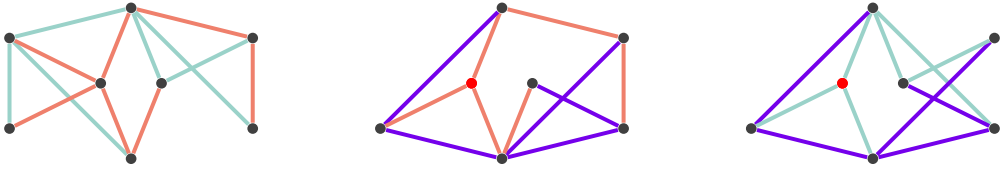
\begin{center}
\begin{tikzpicture}[scale=2]          
	\node[vertex] (r1) at (0,0.5) {};
	\node[vertex] (r2) at (0,-0.5) {};
	\node[vertex] (a1) at (-0.8,0.3) {};
	\node[vertex] (a2) at (0.8,0.3) {};
	\node[vertex] (b1) at (-0.2,0) {};
	\node[vertex] (b2) at (0.2,0) {};
	\node[vertex] (c1) at (-0.8,-0.3) {};
	\node[vertex] (c2) at (0.8,-0.3) {};
	\draw[bedge] (r1)edge(a1);
    \draw[oedge] (r1)edge(b1);
    \draw[oedge] (r1)edge(a2);
    \draw[bedge] (r1)edge(b2);
    \draw[bedge] (r1)edge(c2);
	\draw[bedge] (r2)edge(a1);
    \draw[oedge] (r2)edge(b1);
    \draw[oedge] (r2)edge(b2);
	\draw[oedge] (a1)edge(b1);
    \draw[bedge] (a1)edge(c1);
    \draw[oedge] (b1)edge(c1);
	\draw[bedge] (a2)edge(b2);
    \draw[oedge] (a2)edge(c2);
\end{tikzpicture}\qquad\qquad
\begin{tikzpicture}[scale=2]          
	\node[vertex] (r1) at (0,0.5) {};
	\node[vertex] (r2) at (0,-0.5) {};
	\node[rvertex] (a1) at (-0.2,0) {};
	\node[vertex] (a2) at (0.8,0.3) {};
	\node[rvertex] (b1) at (-0.2,0) {};
	\node[vertex] (b2) at (0.2,0) {};
	\node[vertex] (c1) at (-0.8,-0.3) {};
	\node[vertex] (c2) at (0.8,-0.3) {};
    \draw[oedge] (r1)edge(b1);
    \draw[pedge] (r1)edge(c1);
    \draw[oedge] (r1)edge(a2);
    \draw[oedge] (r2)edge(b1);
    \draw[pedge] (r2)edge(c1);
    \draw[pedge] (r2)edge(a2);
    \draw[oedge] (r2)edge(b2);
    \draw[pedge] (r2)edge(c2);
    \draw[oedge] (b1)edge(c1);
    \draw[oedge] (a2)edge(c2);
    \draw[pedge] (b2)edge(c2);
\end{tikzpicture}\qquad\qquad
\begin{tikzpicture}[scale=2]          
	\node[vertex] (r1) at (0,0.5) {};
	\node[vertex] (r2) at (0,-0.5) {};
	\node[rvertex] (a1) at (-0.2,0) {};
	\node[vertex] (a2) at (0.8,0.3) {};
	\node[rvertex] (b1) at (-0.2,0) {};
	\node[vertex] (b2) at (0.2,0) {};
	\node[vertex] (c1) at (-0.8,-0.3) {};
	\node[vertex] (c2) at (0.8,-0.3) {};
	\draw[bedge] (r1)edge(a1);
    \draw[pedge] (r1)edge(c1);
    \draw[bedge] (r1)edge(b2);
    \draw[bedge] (r1)edge(c2);
	\draw[bedge] (r2)edge(a1);
    \draw[pedge] (r2)edge(c1);
    \draw[pedge] (r2)edge(a2);
    \draw[pedge] (r2)edge(c2);
    \draw[bedge] (a1)edge(c1);
	\draw[bedge] (a2)edge(b2);
    \draw[pedge] (b2)edge(c2);
\end{tikzpicture}
\end{center}
Again, we note here that each of these graphs are $(2,3)$-tight and hence minimally 2-rigid.

As all other edges of the double banana can be formed from one of the above two edges using an automorphism,
the graph satisfies the necessary conditions outlined in \Cref{mainthm}.
While the double banana is $(3,6)$-tight, it is not 3-rigid: for example, it has a 2-vertex separating set that can act like a hinge for one half of the graph to rotate around.
Hence, the converse of \Cref{mainthm} is false.

\section{Alternative proof of \texorpdfstring{\Cref{mainthm}}{Theorem 1.1}}

In this section we provide an alternative proof of \Cref{mainthm}. Before we prove \Cref{mainthm},
we require the following lemmas.
Our first lemma is the following folklore statement regarding Laplace expansions:

\begin{lemma}\label{lem:laplaceexpansion}
    Let $X$ be an $n \times n$ matrix and let $C \subset [n]$ be a set containing $k$ elements.
    Then, given $X_{A,B}$ is the submatrix of $X$ with rows $A$ and columns $B$, we have    
    \begin{equation*}
        \det X = \sum_{R \in \binom{[n]}{k}} (-1)^{\sum_{r \in R} r + \sum_{c \in C} c} \det X_{R,C} \det X_{R^c, C^c}.
    \end{equation*}
    Hence, if $X$ is invertible, there exists $R \in \binom{[n]}{k}$ such that both matrices $X_{R,C}$ and $X_{R^c, C^c}$ are invertible.
\end{lemma}

Each graph $G=(V,E)$ will now be assumed to come equipped with a total ordering $<$ for its vertices.
We additionally will assume that the columns of $R(G,p)$ are ordered such that $(v,i) < (w,j)$ if and only if $i < j$ or $i=j$ and $v < w$.
We will be more loose with our row ordering for $R(G,p)$ since we often wish to switch this throughout.

\begin{lemma}\label{lem1}
    Let $G=(V,E)$ be a minimally 3-rigid graph with $|V|\geq 3$.
    Let $T=(V,F)$ be a spanning tree of $G$ containing edge $e=xy$.
    Then there exists a partition $R_1,R_2$ of $E \setminus F$ with $|R_1|=|V|-2$ and $|R_2|=|V|-3$ such that the graphs
    \begin{equation*}
        (V,F \cup R_1), \qquad (V, R_2 \cup F)/e
    \end{equation*}
    are both minimally 2-rigid.
\end{lemma}

\begin{proof}
    Choose a generic realisation $\tilde{p}$ of $G$ and pick any vertex $z \in V \setminus \{x,y\}$.
    From this framework, we form the framework $(G,p)$ by rotating and translating the framework $(G,\tilde{p})$ until $p(x)-p(y) = (\lambda,0,0)$ for some $\lambda \neq 0$ and $p(z) = (0,0,0)$. 
    Order the rows of $R(G,p)$ so that the rows corresponding to edges in $F$ are at the top.
    We now form the square matrix $X$ from $R(G,p)$ by deleting the columns labelled
    \begin{equation*}
        (x,3), \qquad (y,2), \qquad (y,3), \qquad (z,1), \qquad (z,2), \qquad (z,3).
    \end{equation*}

    \begin{claim}\label{claim1}
        The matrix $X$ is invertible.
    \end{claim}

    \begin{proof}
        Let $u = (u_1(v),u_2(v),u_3(v))_{v \in V} \in \mathbb{R}^{3|V|}$ be a vector in the kernel of $R(G,p)$ satisfying the following equalities:
        \begin{equation*}
            u_3(x)=0, \qquad u_2(y) = 0, \qquad u_3(y)=0, \qquad u(z) = (0,0,0).
        \end{equation*}
        Since $(G,p)$ is infinitesimally rigid,
        there exists a skew-symmetric matrix $A \in \mathbb{R}^{3 \times 3}$ such that $u(v) = Ap(v)$ for each $v \in V$;
        see \cite[\S 2.3]{GraverServatius} for more details.
        By genericity, the points $p(x),p(y)$ are linearly independent.
        Since $Ap(x) = Ap(y) = (0,0,0)$,
        the rank of $A$ is either 0 or 1.
        As all skew-symmetric matrices have even rank (see, for example, \cite[\S 10.3, Thm 6]{linalg}),
        $A$ is the all-zeroes matrix.
        Thus $u$ is all the all-zeroes vector for $\mathbb{R}^{3|V|}$.
        From this it follows that 
        \begin{equation*}
            \rank X = \rank R(G,p) = 3|V|-6.
        \end{equation*}
        The result now holds since $X$ has exactly $3|V|-6$ rows.
    \end{proof}
    
    We can now represent the matrix $X$ as follows:
    \begin{equation*}
        X = 
        \begin{pmatrix}
            X_{1,1} & X_{1,2} & X_{1,3} \\
            X_{2,1} & X_{2,2} & X_{2,3}
        \end{pmatrix},
    \end{equation*}
    where each matrix $X_{1,i}$ has rows labelled by $F$ and columns labelled by $(v,i)$ for suitable $v$,
    and each matrix $X_{2,i}$ has rows labelled by $E \setminus F$ and columns labelled by $(v,i)$ for suitable $v$.
    We observe here that each matrix $X_{i,j}$ has $|V|-j$ columns.
    Moreover, since $p$ is generic and $F$ is a tree, the matrix $X_{1,1}$ is invertible.
    
    We can now perform a block row operation to cancel out the bottom left matrix:
    specifically, we subtract $X_{2,1} X_{1,1}^{-1}$ times row 1 from row 2 to get the matrix
    \begin{equation*}
        X' := 
        \begin{pmatrix}
            X_{1,1} & X_{1,2} & X_{1,3} \\
            \mathbf{0}_{|E\setminus F| \times |V|-1} & X_{2,2} - X_{2,1} X_{1,1}^{-1} X_{1,2} & X_{2,3}  - X_{2,1} X_{1,1}^{-1} X_{1,3}
        \end{pmatrix}.
    \end{equation*}
    Now set
    \begin{equation*}
        Y = 
        \begin{pmatrix}
            Y_1 & Y_2
        \end{pmatrix}
        :=
        \begin{pmatrix}
            X_{2,2} - X_{2,1} X_{1,1}^{-1} X_{1,2} & X_{2,3}  - X_{2,1} X_{1,1}^{-1} X_{1,3}
        \end{pmatrix}.
    \end{equation*}
    Since both $X$ and $X_{1,1}$ are invertible,
    the matrix $Y$ is invertible.

    By \Cref{lem:laplaceexpansion}, there exists a partition $R_1,R_2$ of $E \setminus F$ with $|R_1|=|V|-2$ and $|R_2|=|V|-3$ such that the matrices $Y_1|_{R_1}$ and $Y_2|_{R_2}$ (with the subscript representing the rows each matrix $Y_i$ is restricted to) are invertible.
    Observe here that $Y_i|_{R_i} = X_{2,i+1}|_{R_i} - X_{2,1}|_{R_1} X_{1,1}^{-1} X_{1,i+1}$ for each $i \in \{1,2\}$, and so the matrices
    \begin{equation*}
        \begin{pmatrix}
            X_{1,1} & X_{1,2} \\
            \mathbf{0}_{|R_1| \times |V|-1} & X_{2,2}|_{R_1} - X_{2,1}|_{R_1} X_{1,1}^{-1} X_{1,2}
        \end{pmatrix},
        \qquad 
        \begin{pmatrix}
            X_{1,1} & X_{1,3} \\
            \mathbf{0}_{|R_1| \times |V|-1} & X_{2,3}|_{R_2} - X_{2,1}|_{R_2} X_{1,1}^{-1} X_{1,3}
        \end{pmatrix}
    \end{equation*}
    are invertible.
    Hence, the matrices
    \begin{equation*}
        A =
        \begin{pmatrix}
            X_{1,1} & X_{1,2} \\
            X_{2,1}|_{R_1} & X_{2,2}|_{R_1}
        \end{pmatrix},
        \qquad 
        B =
        \begin{pmatrix}
            X_{1,1} & X_{1,3} \\
            X_{2,1}|_{R_2} & X_{2,3}|_{R_2}
        \end{pmatrix}
    \end{equation*}
    are invertible.

    Set $G_{1} := (V,F \cup R_1)$.
    Fix $q :V \rightarrow \mathbb{R}^2$ to be the realisation of $G_1$ where $q(v) = (p_1(v),p_{2}(v))$ for each $v \in V$.
    We immediately note that $A$ is the matrix formed from $R(G_{1},q)$ by deleting the columns $(y,2),(z,1),(z,2)$.
    Since $A$ is invertible,
    it follows that $\rank R(G_1,q) = 2|V|-3$, and hence $G_1$ is minimally 2-rigid.
    
    Now set $G_2 := (V, F \cup R_2) /e$;
    here we assume the vertex set $G_2$ is $V \setminus x$.
    Before we try the same trick for $B$, we need to slightly adjust it.
    First, add column $(x,1)$ to column $(y,1)$ to form a new invertible matrix $\tilde{B}$.
    With this, the row corresponding to the edge $e$ now has a single non-zero entry in column $(x,1)$ (namely $\lambda$).
    Hence, the matrix $B'$ formed by deleting row $e$ and column $(x,1)$ is also invertible.
    With this, the left-hand side of $B'$ now is a row-scaled copy of the directed incidence matrix of $G_2$ with the column for $z$ deleted,
    and the right-hand side of $B'$ now is a row-scaled copy of the directed incidence matrix of $G_2$ with the columns for $y$ and $z$ deleted.
    
    Now fix $r :V \setminus x \rightarrow \mathbb{R}^2$ to be the realisation of $G_2$ where $r(v) = (p_1(v),p_{3}(v))$ for each $v \in V \setminus x$.
    We immediately note that $B'$ is the matrix formed from $R(G_{2},q)$ by deleting the columns $(y,2),(z,1),(z,2)$.
    Since $B'$ is invertible,
    it follows that $\rank R(G_2,q) = 2|V \setminus x|-3$, and hence $G_2$ is also minimally 2-rigid.
\end{proof}

\begin{lemma}\label{lem2}
    Let $G=(V,E)$ be a minimally 3-rigid graph with edge $e=xy$.
    Then there exists a spanning tree $T=(V,F)$ such that $e \in F$ and the graph 
    \begin{equation*}
        (V, (E \setminus F) \cup e)/e
    \end{equation*}
    is minimally 2-rigid.
\end{lemma}

\begin{proof}
    Choose a generic realisation $\tilde{p}$ of $G$ and pick any vertex $z \in V \setminus \{x,y\}$.
    From this framework, we form the framework $(G,p)$ by rotating and translating the framework $(G,\tilde{p})$ until $p(x)-p(y) = (\lambda,0,0)$ for some $\lambda \neq 0$ and $p(z) = (0,0,0)$. 
    Order the rows of $R(G,p)$ so that the row corresponding to the edge $e$ is at the top.

    Pick any vertex $z \in V \setminus \{x,y\}$.
    We now form the square matrix $Z$ from $R(G,p)$ by deleting the columns labelled
    \begin{equation*}
        (x,3), \qquad (y,3), \qquad (z,1), \qquad (z,2), \qquad (z,3)
    \end{equation*}
    and adding a row (now labelled $\tilde{e}$ and considered to be a parallel edge to $e$) with value 1 in the column $(y,2)$ and value 0 elsewhere.
    Given $X$ is the matrix described in \Cref{lem1},
    we observe that $|\det Z| = |\det X|$.
    Hence, the matrix $Z$ is invertible by \Cref{claim1}.

    Consider the block representation $Z = (Z_1 ~ Z_2 ~ Z_3)$ where each matrix $Z_i$ has all columns with index $(v,i)$ for some vertex $v$.
    By \Cref{lem:laplaceexpansion},
    there exists $|V|-1$ rows $S \subset E \cup \tilde{e}$ so that the matrices 
    \begin{equation*}
        B=
        \begin{pmatrix}
            Z_1|_{E \setminus S} & Z_3|_{E \setminus S}
        \end{pmatrix},
        \qquad Z_2|_S
    \end{equation*}
    are both invertible.
    Since the row for $e$ has all zeroes in $Z_2$,
    we have that $e \notin S$.
    Similarly, since the row for $\tilde{e}$ has all zeroes in $Z_1,Z_3$, we have that $\tilde{e} \in S$.
    Now set $F = (S \setminus \tilde{e}) \cup e$.
    The matrix $X_2|_S$ is formed from taking a copy of the directed incidence matrix of the graph $(V,F)$, scaling each row by a non-zero value (with non-zero following from $p$ being constructed carefully from a generic framework) and deleting the column for the vertex $z$.
    Since $X_2|_S$ is invertible, it now follows that $(V,F)$ is a tree.

    We now proceed with a similar trick to that described in the proof of \Cref{lem1}.
    Fix $H = (V, (E \setminus F) \cup e) /e$;
    here we assume the vertex set $H$ is $V \setminus x$.
    First, add column $(x,1)$ to column $(y,1)$ to form a new invertible matrix $\tilde{B}$.
    With this, the row corresponding to the edge $e$ now has a single non-zero entry in column $(x,1)$ (namely $\lambda$).
    Hence, the matrix $B'$ formed by deleting row $e$ and column $(x,1)$ is also invertible.
    With this, the left-hand side of $B'$ now is a row-scaled copy of the directed incidence matrix of $H$ with the column for $z$ deleted,
    and the right-hand side of $B'$ now is a row-scaled copy of the directed incidence matrix of $H$ with the columns for $y$ and $z$ deleted.
    
    Now fix $r :V \setminus x \rightarrow \mathbb{R}^2$ to be the realisation of $H$ where $r(v) = (p_1(v),p_{3}(v))$ for each $v \in V \setminus x$.
    We immediately note that $B'$ is the matrix formed from $R(H,q)$ by deleting the columns $(y,2),(z,1),(z,2)$.
    Since $B'$ is invertible,
    it follows that $\rank R(H,q) = 2|V \setminus x|-3$, and hence $H$ is minimally 2-rigid.
\end{proof}

With this, we can now prove \Cref{mainthm}:

\begin{proof}[Proof of \Cref{mainthm}]
    Set $T=(V,F)$ to be the spanning tree given by \Cref{lem2} for $G$ and $e$,
    and set $R_1,R_2$ to be the partition of $F \setminus E$ given by \Cref{lem1} for $G$, $T$ and $e$.
    The result now follows by setting $S_1 = F$, $S_2=R_1$ and $S_3=R_2$.
\end{proof}

\section{Final thoughts}

In \cite{baranyai2026genericrigiditygraphs},
Baranyai additionally claims to prove a combinatorial characterisation of minimal $d$-rigidity; see \cite[Theorem 2]{baranyai2026genericrigiditygraphs}.
The author believes that the counter-example presented in \Cref{sec:counter} can be extended to a counter-example for the higher-dimensions by \emph{coning} the graph;
i.e., adding a vertex adjacent to all other vertices.
Here we recall that the coning of a graph is (minimally) $(d+1)$-rigid if and only if the original graph is (minimally) $d$-rigid;
see, for example, \cite{Whiteley1983}.

In the author's opinion, the counter-example presented in \Cref{sec:counter} motivates the following conjecture.

\begin{conjecture}\label{mainconj}
    A graph $G=(V,E)$ with at least 4 vertices is $(3,6)$-tight if and only if for any edge $e  \in E$, there exists a partition $(S_1,S_2,S_3)$ of $E$ such that the following holds:
    \begin{enumerate}
        \item $|S_i|=|V|-i$ for each $i \in \{1,2,3\}$.
        \item $e \in S_1$.
        \item Each of the graphs
        \begin{equation*}
            (V, S_1 \cup S_2) , \qquad (V, S_1 \cup S_3) /e, \qquad (V, S_2 \cup S_3 \cup e)/e
        \end{equation*}
        is $(2,3)$-tight.
    \end{enumerate}
\end{conjecture}

\subsection*{Acknowledgements}

The author would like to thank Matteo Gallet and Oliver Clarke for their discussion of the paper \cite{baranyai2026genericrigiditygraphs}.

S.\,D.\ was supported by the FWO grants G0F5921N (Odysseus) and G023721N, and by the KU Leuven grant iBOF/23/064.

\bibliographystyle{plainurl}
\bibliography{ref}
\end{document}